\documentclass[11pt]{amsart}
\usepackage{cleveref,amssymb,dsfont,tikz}
\usetikzlibrary{arrows}

\author{Helen Samara Dos Santos}
\address{Department of Mathematics and Statistics, Memorial University of Newfoundland, St. John's, NL, A1C5S7, Canada.}
\email[H.S.~Dos Santos]{helensds@mun.ca}
\author{Felipe Yukihide Yasumura}
\address{Department of Mathematics, Instituto de Matem\'atica e Estat\'istica, Universidade de S\~ao Paulo, SP, Brazil.}
\email[F.~Yasumura]{fyyasumura@ime.usp.br}
\thanks{The second named author is supported by Fapesp, grants no.~2023/03922-8 and no.~2018/23690-6}


\title{Group gradings on finite-dimensional incidence algebras. II}

\newtheorem{Thm}{Theorem}[section]
\newtheorem{Prop}[Thm]{Proposition}
\newtheorem{Lemma}[Thm]{Lemma}
\newtheorem{Cor}[Thm]{Corollary}
\theoremstyle{definition}
\newtheorem{Def}[Thm]{Definition}
\theoremstyle{remark}
\newtheorem{Remark}[Thm]{Remark}

\newcommand{\YD}[2]{{}_{\mathbb{F}#1}\mathcal{YD}^{\mathbb{F}#2}}
\newcommand{\up}{\!\uparrow}

\begin{document}
\begin{abstract}
We complete the description of group gradings on finite-dimensional incidence algebras. Moreover, we classify the finite-dimensional graded algebras that can be realized as incidence algebras endowed with a group grading.
\end{abstract}
\maketitle

\section{Introduction}
The classification of gradings by arbitrary groups on a given algebra has been a topic of significant interest, particularly after the groundbreaking works by Patera and Zassenhaus \cite{PZ} and Bahturin, Sehgal, and Zaicev \cite{BSZ}. Such classifications have been obtained for many important classes of algebras, including the simple finite-dimensional associative, Lie, and Jordan algebras (see the monograph \cite{EK13} and the references therein).

Advancements in the classification of group gradings on non-simple algebras have also been made, as seen in the articles \cite{BesDas, DGK, KY}. In particular, the complete classification of isomorphism classes of group gradings on the algebra of upper triangular matrices is given in the works \cite{VinKoVa2004, VaZa2007}. As a generalization of these algebras, the description of group gradings on incidence algebras is provided in the paper \cite{SSY}, and their isomorphism classes are described when the grading group is abelian. Their main result states that every group grading on a finite-dimensional incidence algebra is graded isomorphic to a graded upper triangular algebra. Although the authors characterize the structure of graded bimodules, there is no description of the algebra structure of the upper triangular algebra.

In this paper, we complete the work initiated in \cite{SSY} and, in particular, answer the last question of the paper. We provide a complete description of group gradings on a given incidence algebra. Moreover, we classify the finite-dimensional graded algebras that can be realized as an incidence algebra endowed with a group grading. As a byproduct, we investigate the structure of graded bimodules and elucidate a connection between Yetter-Drinfel'd modules and graded bimodules.

This paper is organized as follows. In \Cref{S_Preliminaries}, we give the basic definitions concerning group gradings, incidence algebras, and the description of group gradings on incidence algebras given in \cite{SSY}. In \Cref{S_grBimodules}, we investigate the structure of graded bimodules, proposing a new approach to the theory. The following section is devoted to classifying the graded triangular algebras that can be realized as an incidence algebra endowed with a group grading (\Cref{realization}). An explicit construction of the poset is also provided (\Cref{construction}). The last section is dedicated to studying the product of bimodules that occurs in an incidence algebra. We prove that the characters appearing in a product of bimodules should be extensions of the product of the characters (\Cref{mainresult}). Finally, in the last section, we summarize the results obtained, characterizing the graded algebras that can be realized as graded incidence algebras (\Cref{realizationgradedalgebra}).

\section{Preliminaries\label{S_Preliminaries}}
\subsection{Group gradings} 
Let $G$ be a group and $\mathcal{A}$ an $\mathbb{F}$-algebra. We will use the multiplicative notation for the group $G$ and denote its neutral element by $1$. A \emph{$G$-grading} on $\mathcal{A}$ is a vector space decomposition
$$
\mathcal{A}=\bigoplus_{g\in G}\mathcal{A}_g,
$$
such that $\mathcal{A}_g\mathcal{A}_h\subseteq\mathcal{A}_{gh}$ for all $g, h \in G$. If $\mathcal{A}$ has a fixed $G$-grading, we say that $\mathcal{A}$ is $G$-graded. The component $\mathcal{A}_g$ is called the homogeneous component of degree $g$, and its nonzero elements are said to be \emph{homogeneous} of degree $g$. Given $x\in\mathcal{A}_g$ with $x \ne 0$, we denote $\deg_G x=g$. A subspace $\mathcal{S}\subseteq\mathcal{A}$ is called \emph{graded} if $\mathcal{S}=
\bigoplus_{g\in G} \mathcal{S}\cap\mathcal{A}_g$. A graded subalgebra (or ideal) is a subalgebra (or ideal) that is a graded subspace. The \emph{support} of the graded algebra $\mathcal{A}$ is $\mathrm{Supp}\, \mathcal{A}=\{g\in G\mid\mathcal{A}_g\ne0\}$.

If $\mathcal{B}$ is another $G$-graded algebra, then a \emph{homomorphism of $G$-graded algebras} is a homomorphism of algebras $\varphi:\mathcal{A}\to\mathcal{B}$ such that $\varphi(\mathcal{A}_g)\subseteq\mathcal{B}_g$ for all $g\in G$. If $\varphi$ is an algebra isomorphism, then $\mathcal{A}$ and $\mathcal{B}$ are said to be \emph{$G$-graded isomorphic}, denoted by $\mathcal{A}\cong_G\mathcal{B}$. A complete reference on the subject of graded algebras is the monograph \cite{EK13}.

Let $\mathcal{A}$ and $\mathcal{B}$ be $G$-graded algebras. Let $M$ be an $(\mathcal{A},\mathcal{B})$-bimodule, and assume that $M$ has a $G$-grading, say $M =\bigoplus_{g\in G} M_g$. We say that $M$ is a \emph{$G$-graded $(\mathcal{A},\mathcal{B})$-bimodule} if $\mathcal{A}_h M_g \mathcal{B}_k \subseteq M_{hgk}$ for all $h, g, k \in G$. If $N$ is another $G$-graded $(\mathcal{A},\mathcal{B})$-bimodule, a \emph{$G$-graded homomorphism of bimodules} is a bimodule homomorphism $f:M\to N$ such that $f(M_g)\subseteq N_g$ for all $g\in G$. If $f$ is bijective, we say that $M$ and $N$ are \emph{$G$-graded isomorphic} and denote $M \cong_G N$.
    
A \emph{graded triangular algebra} is an algebra of the form
\[
\left(\begin{array}{cc}\mathcal{A}&M\\&\mathcal{B}\end{array}\right)
\]
where $\mathcal{A}$ and $\mathcal{B}$ are $G$-graded algebras, and $M$ is a $G$-graded $(\mathcal{A},\mathcal{B})$-bimodule.

\subsection{Incidence algebras} We provide the definition of an incidence algebra over a field $\mathbb{F}$. Let $(X,\le)$ be any partially ordered set (poset, for short). Assume that $(X,\le)$ is \emph{locally finite}, i.e., for all $x,y\in X$, there exists a finite number of $z\in X$ such that $x\le z\le y$. Define $I(X)=\{f:X\times X\to \mathbb{F}\mid f(x,y)=0,\forall x\not\le y\}$. Then $I(X)$ has a natural sum (point-wise sum) and natural scalar multiplication, which give $I(X)$ the structure of an $\mathbb{F}$-vector space. For $f,g\in I(X)$, we define $h=f\cdot g$ as the function $h$ such that $h(x,y)=\sum_{z\in X} f(x,z)g(z,y)$. Note that the only possibly nonzero elements in the previous sum are the $z\in X$ such that $x\le z\le y$; hence, since $X$ is locally finite, the sum is well-defined. So $h\in I(X)$. It is straightforward to prove that $I(X)$, with the defined operations, is an associative algebra. The algebra $I(X)$ is called an \emph{incidence algebra}.

Given $x\le y$, we let $e_{xy}$ be the element such that
\[
e_{x,y}(w,z)=\left\{\begin{array}{ll}1,&\text{ if $x=w$ and $y=z$},\\0,&\text{ otherwise}\end{array}\right.
\]
Note that $e_{xy}e_{wz}=\delta_{yw}e_{xz}$. We will denote $e_x:=e_{xx}$.

Note that the defined multiplication on $I(X)$ is similar to the product of matrices. Moreover, if $X$ is totally ordered and contains $n$ elements, then $I(X)\cong\mathrm{UT}_n$, the algebra of upper triangular matrices. In connection, if $(X,\le_X)$ is arbitrary (not necessarily totally ordered) and finite with $n$ elements, then we can rename the elements of $X=\{1,2,\ldots,n\}$ in such a way that $i\le_X j$ implies $i\le j$ in the usual ordering of the integers. With this identification, we see that $I(X)\subset\mathrm{UT}_n$ is a subalgebra. Thus, finite-dimensional incidence algebras are subalgebras of $\mathrm{UT}_n$ containing all the diagonal matrices. As we are interested in finite-dimensional incidence algebras, we will assume from now on that $I(X)\subset\mathrm{UT}_n$.

The incidence algebras are very interesting on their own, and moreover, they are related to other branches of Mathematics. They also give rise to interesting and challenging combinatorial problems. For an extensive theory on incidence algebras, see, for instance, the book \cite{SpieDon}.

\subsection{Group gradings on incidence algebras} In this subsection, we recall the main results of \cite{SSY}. It describes the group gradings on finite-dimensional incidence algebras. More precisely, they prove:

\begin{Thm}[{\cite[Theorem1]{SSY}}]
Let $\mathbb{F}$ be a field, $X$ a finite poset, and let $I(X)$ be endowed with a $G$-grading. Assume at least one of the following conditions: $\mathrm{char}\,\mathbb{F}=0$, $\mathrm{char}\,\mathbb{F}>\dim I(X)$, or $G$ is abelian. Then, up to a graded isomorphism, there exist finite abelian subgroups $H_1,\ldots,H_t\subseteq G$, such that: for each $i=1,\ldots,t$, $\mathrm{char}\,\mathbb{F}$ does not divide $|H_i|$ and $\mathbb{F}$ contains a primitive $\mathrm{exp}\,H_i$-root of $1$, and
\[
I(X)\cong_G\left(\begin{array}{cccc}\mathbb{F}H_1&M_{1,2}&\ldots&M_{1,t}\\&\mathbb{F}H_2&\ddots&\vdots\\&&\ddots&M_{t-1,t}\\&&&\mathbb{F}H_t\end{array}\right),
\]
where each $M_{i,j}$ is a $G$-graded $(\mathbb{F}H_i,\mathbb{F}H_j)$-bimodule.
\end{Thm}

An isomorphism condition of group gradings is also provided \cite[Proposition 31]{SSY}, and a complete answer is given when the grading group is finite (\cite[Theorem 34]{SSY}). Moreover, the authors provide a description of the graded bimodules (see the section below as well). However, they do not describe the algebra structure on the triangular algebra, nor do they specify which triangular algebras can be realized as incidence algebras endowed with a group grading. In this paper, we answer these questions.

\section{Graded Bimodules\label{S_grBimodules}}
In this section, we investigate graded bimodules. The results presented here are known (see \cite[Section 4]{SSY} and also \cite{FS}). However, we propose a new approach that may shed light on the case where the grading group is non-abelian.

\begin{Def}
Let $G$ be an arbitrary group, $H \subseteq G$ a finite subgroup, $\mathcal{V}$ a $\mathbb{F}$-vector space, $\rho:H\to\mathrm{GL}(\mathcal{V})$ a representation of $H$ on $\mathcal{V}$, and $\Gamma:\mathcal{V}=\bigoplus_{g\in G}\mathcal{V}_g$ a $G$-grading on $\mathcal{V}$. We say that $\rho$ and $\Gamma$ are compatible if $\rho(h)(\mathcal{V}_g)\subseteq\mathcal{V}_{hgh^{-1}}$ for all $h\in H$ and $g\in G$.
\end{Def}

\begin{Remark}
The previous definition is equivalent to the following. The vector space $\mathcal{V}$ is a left $\mathbb{F}H$-module and a right $\mathbb{F}G$-comodule, and these structures satisfy the compatibility condition of a Yetter-Drinfel'd module. Denote by $\YD{H}{G}$ the category in which the objects are Yetter-Drinfel'd left $\mathbb{F}H$ modules and right $\mathbb{F}G$-comodules, and the morphisms are homomorphisms of $\mathbb{F}H$-modules and $\mathbb{F}G$-comodules. Therefore, from now on, we will simply say that $\mathcal{V}\in\YD{H}{G}$ instead of saying that $\mathcal{V}$ has a pair of compatible $G$-grading and representation of $H$.
\end{Remark}

The following results are known:
\begin{Lemma}\label{gradedMashcke}
Let $\mathcal{V}\in\YD{H}{G}$. Then:
\begin{enumerate}
\renewcommand{\labelenumi}{(\roman{enumi})}
\item $\mathcal{V}$ is completely reducible.
\item For any $g\in\mathrm{Supp}\,\mathcal{V}$, the restriction $\rho_g:C_H(g)\to\mathrm{GL}(\mathcal{V}_g)$ is a representation and $\mathcal{V}=\mathrm{Ind}_{C_H(g)}^H\rho_g$.
\item $\mathcal{V}$ is irreducible (as Yetter-Drinfel'd module) if, and only if, $\rho_g$ is irreducible.
\end{enumerate}
\end{Lemma}

Given two subgroups $H_1$, $H_2$, let $H_{12}=H_1\cap H_2$. Then, $\mathbb{F}(H_1\times H_2)$ is a right $\mathbb{F}H_{12}$-module via $(h_1,h_2)\cdot h=(h_1h,h^{-1}h_2)$. The Yetter-Drinfel'd modules relate to graded bimodules in the following way.

\begin{Prop}\label{construction_bimodule}
Let $G$ be a finite group, $H_1$, $H_2\subseteq G$ be finite groups, $H_{12}=H_1\cap H_2$ and $\mathcal{V}\in\YD{H_{12}}{G}$. Then
$$
\mathcal{V}\up:=\mathbb{F}(H_1\times H_2)\otimes_{\mathbb{F}H_{12}}\mathcal{V}
$$
is a $G$-graded $(\mathbb{F}H_1,\mathbb{F}H_2)$-bimodule via:
\begin{enumerate}
\item $h\cdot((h_1,h_2)\otimes v)\cdot k=(hh_1,h_2k)\otimes v$, $h\in H_1$, $k\in H_2$, $(h_1,h_2)\otimes v\in\mathcal{V}\up$.
\item $\deg((h_1,h_2)\otimes v)=h_1(\deg v)h_2$.
\end{enumerate}
\end{Prop}

\begin{proof}
Given $h\in H_1$, define
\[
L_h:\mathbb{F}(H_1\times H_2)\times\mathcal{V}\to\mathcal{V}\up
\]
via $L_h((h_1,h_2),v)=(hh_1,h_2)\otimes v$. We need to prove that $L_h$ is $\mathbb{F}H_{12}$-balanced. Given $t\in H_{12}$, one has
\begin{align*}
L_h((h_1,h_2)\cdot t,v)&=L_h((h_1t,t^{-1}h_2),v)=(hh_1t,t^{-1}h_2)\otimes v=(hh_1,h_2)\otimes t\cdot v\\&=L_h((h_1,h_2),t\cdot v).
\end{align*}
Thus, we obtain a left multiplication by $h$, $L_h:\mathcal{V}\up\to\mathcal{V}\up$ as in the statement. Similarly, we obtain a right multiplication by $k\in H_2$. It is clear that the defined operations give a structure of bimodule on $\mathcal{V}\up$.

We will prove that the given degree is well-defined. For this, it is enough to show that $\deg(((h_1,h_2)\cdot h)\otimes v))=\deg((h_1,h_2)\otimes h\cdot v)$. The first one equals $\deg((h_1h,h^{-1}h_2)\otimes v)=h_1h(\deg v)h^{-1}h_2$. On the other hand, the second one equals $h_1(\deg(h\cdot v))h_2=h_1h(\deg v)h^{-1}h_2$. Hence, the degree is well-defined. From the definition of the bimodule action, it is clear that $\mathcal{V}\up$ is a graded bimodule.
\end{proof}

Note that we can identify $\mathcal{V}=(1,1)\otimes\mathcal{V}\subseteq\mathcal{V}\up$. Moreover, we can recover the structure of $\mathbb{F}H_{12}$-module in terms of the bimodule structure: given $h\in H_{12}$, one has $h\cdot v=h\cdot((1,1)\otimes v)\cdot h^{-1}$.

It would be interesting to prove that $\mathcal{V}\up$ is irreducible whenever $\mathcal{V}$ is irreducible. We can prove this fact if we impose some extra conditions. The converse is trivially valid.

\begin{Lemma}\label{lem3}
Let $G$ be a group, $H_1$, $H_2\subseteq G$ finite subgroups, and $\mathcal{V}$, $\mathcal{W}\in\YD{H_{12}}{G}$. Then
\begin{enumerate}
\renewcommand{\labelenumi}{(\roman{enumi})}
\item $(\mathcal{V}\oplus\mathcal{W})\up=\mathcal{V}\up\oplus\mathcal{W}\up$.
\item If $\mathcal{V}\up$ is irreducible, then so is $\mathcal{V}$.
\item Any morphism $\mathcal{V}\to\mathcal{W}$ extends to a unique $G$-graded bimodule homomorphism $\mathcal{V}\up\to\mathcal{W}\up$.
\end{enumerate}
\end{Lemma}

\begin{proof}
The first statement follows from the property of the tensor product. Hence, if $\mathcal{V}$ is reducible, then the same is true for $\mathcal{V}\up$, so we obtain the second statement.

For the third statement, given a morphism $f_0:\mathcal{V}\to\mathcal{W}$, then $f=1\otimes f_0$ is a $G$-graded homomorphism of bimodules. It is unique since $\mathcal{V}\up$ is generated by $\mathcal{V}$ as a bimodule.
\end{proof}

Denote by ${}_{\mathbb{F}H_1}\mathrm{Mod}^{G}_{\mathbb{F}H_2}$ the category of $G$-graded $(\mathbb{F}H_1,\mathbb{F}H_2)$-bimodules. Then, \Cref{construction_bimodule} proves that we can associate a given $\mathcal{V}\in\YD{H_{12}}{G}$ to  $\mathcal{V}\up\in{}_{\mathbb{F}H_1}\mathrm{Mod}^{G}_{\mathbb{F}H_2}$. A morphism $f:\mathcal{V}\to\mathcal{W}$ is associated to $1\otimes f:\mathcal{V}\up\to\mathcal{W}\up$. It is not hard to see that it is a functor, which we will denote by $\uparrow$. Moreover, \Cref{lem3}(iii) shows that the functor is faithful.

Next, we aim to prove that, given certain conditions on the grading group, the functor is essentially surjective. For this purpose, we need:

\begin{Lemma}\label{irreducible}
Assume that $H_{12}\subseteq Z(G)$, the center of $G$, and $H_2$ (or $H_1$) is normal in $G$. Then $\mathcal{V}\up$ is an irreducible graded bimodule if, and only if, $\mathcal{V}$ is irreducible.
\end{Lemma}

\begin{proof}
Assume that $\mathcal{V}$ is irreducible. From \Cref{gradedMashcke}(iii), we can find an irreducible representation from a subgroup of the form $C_{H_{12}}(g)$ that induces $\mathcal{V}$. Since $H_{12}$ is central, it means that $C_{H_{12}}(g)=H_{12}$. Thus, $\mathcal{V}$ is an irreducible representation, so $\dim\mathcal{V}=1$. Denote by $\chi:H_{12}\to\mathrm{GL}(\mathcal{V})$ the action of $H_{12}$ on $\mathcal{V}$.

Let $v\in\mathcal{V}\up$ be a nonzero homogeneous element of degree $g\in G$. We will prove that $v$ generates $\mathcal{V}\up$ as a bimodule. Write
\[
v=\sum_{i=1}^t(h_i,k_i)\otimes v_i,\quad h_i\in H_1,k_i\in H_2,v_i\in\mathcal{V}.
\]
Note that $g=h_i(\deg v_i)k_i$, for each $i=1,\ldots,t$. Since $\dim\mathcal{V}=1$, we can write $v_i=\lambda_iv_1$, for each $i=1,\ldots,t$. In addition, since $v$ is homogeneous, $h_1(\deg v_1)k_1=h_i(\deg v)k_i$. Then,
\[
h_1^{-1}h_i=(\deg v_1)k_1k_i(\deg v_1)^{-1}\in H_1\cap((\deg v_1)H_2(\deg v_2)^{-1})=H_{12},
\]
since $H_2$ is normal in $G$. So, we can write $h_i=h_1h_i'$, for some $h_i'\in H_{12}$. Thus, we have
\begin{align*}
v&=\sum_{i=1}^t(h_i,k_i)\otimes v_i=\sum_{i=1}^t\underbrace{(h_1h_i',h_i^{\prime-1}h_i'k_i)}_{(h_1,h_i'k_i)\cdot h_i'}\otimes(\lambda_iv_1)\\%
&=\sum_{i=1}^t(h_1,h_i'k_i)\otimes(\lambda_i\underbrace{h_i'\cdot v_1}_{\chi(h_i')v_1})=\left(h_1,\sum_{i=1}^t\lambda_i\chi(h_i')h_i'k_i\right)\otimes v_1.
\end{align*}
Since $v$ is homogeneous, one has $h_1(\deg v_1)h_1'k_1=h_1(\deg v_1)h_i'k_i$, for each $i=1,\ldots,t$. So $h_i'k_i=h_1'k_1$, for each $i$. Therefore,
\[
v=\left(h_1,\left(\sum_{i=1}^t\lambda_i\chi(h_i')\right)h_1'k_1\right)\otimes v_1=\left(\sum_{i=1}^t\lambda_i\chi(h_i')\right)\left(h_1,h_1'k_1\right)\otimes v_1.
\]
Hence, $\mathbb{F}H_1v\mathbb{F}H_2=\mathcal{V}\up$. So, $\mathcal{V}\up$ is irreducible.
\end{proof}

As a consequence, if the group $G$ is abelian, then $\mathcal{V}$ is irreducible if, and only if, $\mathcal{V}\up$ is irreducible. From this, we can derive the following description of graded bimodules.

\begin{Thm}\label{thm:gradedbimodule}
Let $G$ be an abelian group, $H_1$, $H_2\subseteq G$ finite subgroups, and $M$ a $G$-graded $(\mathbb{F}H_1,\mathbb{F}H_2)$-bimodule. Then, there exists a unique, up to isomorphism, Yetter-Drinfel'd left $\mathbb{F}H_{12}$-module and right $\mathbb{F}G$-comodule $\mathcal{V}$ such that $M\cong_G\mathcal{V}\up$ (see \Cref{construction_bimodule}). Furthermore, $\mathcal{V}$ is irreducible if, and only if, $\mathcal{V}\up$ is.
\end{Thm}

\begin{proof}
Let $M$ be a $G$-graded $(\mathbb{F}H_1,\mathbb{F}H_2)$-bimodule. Denote by $L_h$ and $R_k$ the left multiplication by $h\in H_1$ and right multiplication by $k\in H_2$, respectively. Then, we obtain a representation of $H_{12}$ on $M$ via $h\mapsto L_h\circ R_{h^{-1}}$. Hence, $M\in\YD{H_{12}}{G}$. From \Cref{gradedMashcke}(i), we can write $M=\bigoplus_{i\in I}\mathcal{V}_i$, where each $\mathcal{V}_i$ is irreducible. We let $\mathbb{F}H_1\mathcal{V}_i\mathbb{F}H_2$ be the graded bimodule generated by $\mathcal{V}_i$. Note that
\[
\mathbb{F}(H_1\times H_2)\times\mathcal{V}_i\to\mathbb{F}H_1\mathcal{V}_i\mathbb{F}H_2,
\]
given by $((h_1,h_2),v)\mapsto h_1\cdot v\cdot h_2$ is a surjective $\mathbb{F}H_{12}$-balanced map, left $\mathbb{F}H_1$-linear and right $\mathbb{F}H_2$-linear. Hence, it induces a surjective $G$-graded bimodule homomorphism $\mathcal{V}_i\up\to\mathbb{F}H_1\mathcal{V}_i\mathbb{F}H_2$. From \Cref{irreducible}, $\mathcal{V}_i\up$ is graded-irreducible, so we obtain a $G$-graded isomorphim of $(\mathbb{F}H_1,\mathbb{F}H_2)$-bimodules $\mathcal{V}_i\up\to\mathbb{F}H_1\mathcal{V}_i\mathbb{F}H_2$. In particular, each $\mathbb{F}H_1\mathcal{V}_i\mathbb{F}H_2$ is graded-simple. Now, it is clear that $M=\sum_{i\in I}\mathbb{F}H_1\mathcal{V}_i\mathbb{F}H_2$. Since an intersection $(\mathbb{F}H_1\mathcal{V}_i\mathbb{F}H_2)\cap(\mathbb{F}H_1\mathcal{V}_j\mathbb{F}H_2)$ is a graded sub-bimodule of a graded-simple bimodule, we see that either $\mathbb{F}H_1\mathcal{V}_i\mathbb{F}H_2=\mathbb{F}H_1\mathcal{V}_j\mathbb{F}H_2$ or their intersection is $0$. Hence, we can find a subset $J$ of $I$ such that $M=\bigoplus_{j\in J}\mathbb{F}H_1\mathcal{V}_j\mathbb{F}H_2$. Therefore, if we define $\mathcal{V}=\bigoplus_{j\in J}\mathcal{V}_j$, one has $\mathcal{V}\in\YD{H_{12}}{G}$ and $M\cong_G\mathcal{V}\up$.

The uniqueness, up to isomorphism, is due to the functor $\YD{H_{12}}{G}\to{}_{\mathbb{F}H_1}\mathrm{Mod}^{G}_{\mathbb{F}H_2}$ being faithful (\Cref{lem3}(iii)). The last statement is proved in \Cref{irreducible}.
\end{proof}

The discussion presented in this section proves the following.
\begin{Cor}
Let $G$ be an abelian group, $H_1$, $H_2\subseteq G$ be finite subgroups, and $H_{12}=H_1\cap H_2$. Then, the functor $\uparrow:\YD{H_{12}}{G}\to{}_{\mathbb{F}H_1}\mathrm{Mod}^{G}_{\mathbb{F}H_2}$ is faithful and essentially surjective on objects.\qed
\end{Cor}

\noindent\textbf{Question.} Is the statement of the previous corollary true if $G$ is not necessarily an abelian group?

As a final remark, assume that $G$ is an abelian group, let $H_1$, $H_2\subseteq G$ be finite subgroups, and $H_{12} = H_1 \cap H_2$. Then, $H_{12}$ is a central subgroup of $G$. Hence, every irreducible $\mathcal{V}\in\YD{H_{12}}{G}$ has dimension $1$ (see the beginning of the proof of \Cref{irreducible}). This means that $\mathcal{V}$ is described by a choice of a character $\chi\in\widehat{{H}_{12}}$ and an element $g\in G$. Thus, a $G$-graded irreducible $(\mathbb{F}H_1,\mathbb{F}H_2)$-bimodule is parameterized by a pair $(\chi,g)$, where $\chi\in\widehat{H_{12}}$ and $g\in G$. Given a finite-dimensional $G$-graded $(\mathbb{F}H_1,\mathbb{F}H_2)$-bimodule $M$, it has a decomposition as a sum of graded-irreducible bimodules, each of which is parameterized by a pair $(\chi,g)$. Hence, we will denote
\[
M\cong_G[(\chi_1,g_1),\ldots,(\chi_t,g_t)].
\]
Given a character $\chi\in\widehat{H_{12}}$, we will denote $\chi\in[M]$ if $\chi=\chi_i$ for some $i\in\{1,2,\ldots,t\}$. In this case, we write $\deg\chi=g_i$.

\begin{Prop}\label{isogradedbimodules}
Let $G$ be an abelian group, $H_1$, $H_2\subseteq G$ be finite groups, $H_{12}=H_1\cap H_2$. Let $M$ and $M'$ be $G$-graded $(\mathbb{F}H_1,\mathbb{F}H_2)$-bimodules, and denote $M\cong_G[(\chi_1,g_1),\ldots,(\chi_t,g_t)]$ and $M'\cong_G[(\chi_1',g_1'),\ldots,(\chi_{t'}',g_{t'}')]$. Then $M\cong_GM'$ (as $G$-graded bimodules) if, and only if, $t=t'$, and there exists a permutation $\sigma\in\mathcal{S}_t$ such that $\chi_j=\chi_{\sigma(j)}'$ and $g_j\in H_1g_{\sigma(j)}'H_2$, for each $j=1,2,\ldots,t$.
\end{Prop}

\begin{proof}
First, we will classify the isomorphism condition for a graded-simple bimodule. Assume that $g\in G$ and $\chi\in\widehat{H_{12}}$. From definition, $[(\chi,g)]\cong\mathcal{V}\up$, where $\mathcal{V}$ is the one-dimensional representation space of $\chi$, and it is homogeneous of degree $g$. We can write $\mathcal{V}\up=\bigoplus_{i=k}^r\mathcal{W}_k$, sum of graded and irreducible $H_{12}$-representations. From the proof of \Cref{irreducible}, one obtains $\mathcal{V}\up=\mathbb{F}H_1\mathcal{W}_k\mathbb{F}H_2$, for each $k=1,2,\ldots,r$. The degree of each subspace $\mathcal{W}_k$ corresponds to some degree of $[(\chi,g)]$, which coincides with $h_1gh_2$, for some $h_1\in H_1$ and $h_2\in H_2$. Hence, $[(\chi,g)]\cong_G[(\chi,h_1gh_2)]$, for any $h_1\in H_1$, $h_2\in H_2$. In particular, the character of each $\mathcal{W}_k$ is $\chi$.

Conversely, assume that $[(\chi,g)]\cong_G[(\chi',g')]$. First, we decompose each of the spaces as a sum of irreducible $H_{12}$-representations. From the reasoning of the previous paragraph, we get $\chi=\chi'$. A $G$-graded isomorphism between the graded bimodules implies that $g\in\mathrm{Supp}[(\chi',g')]$. Thus, $g\in H_1g'H_2$. Hence, $[(\chi,g)]\cong_G[(\chi',g')]$ if, and only if, $\chi=\chi'$ and $g\in H_1g'H_2$.

Now, let $M$ and $M'$ be as in the statement, and assume we have a $G$-graded isomorphism of bimodules $M\to M'$. A standard argument shows that $t=t'$ and there exists $\sigma\in\mathcal{S}_t$ such that $[(\chi_j,g_j)]\cong_G[(\chi_{\sigma(j)},g_{\sigma(j)})]$, for each $j=1,2,\ldots,t$. Hence, from the previous paragraph, $\chi_j=\chi_{\sigma(j)}$ and $g_j\in H_1g_{\sigma(j)}'H_2$, for each $j=1,2,\ldots,t$.
\end{proof}

\section{Realization of graded triangular algebra\label{S_realization}}
In this section, we will construct a realization of graded triangular algebras. More precisely, we prove the following result.
\begin{Thm}\label{construction}
Let $G$ be an abelian group, $H_1$, $H_2\subseteq G$ be finite subgroups, and denote $H_{12}=H_1\cap H_2$. Let $\chi_1$, \ldots, $\chi_t\in\widehat{H_{12}}$, where $\chi_i\ne\chi_j$ for $i\ne j$, and $g_1$, \ldots, $g_t\in G$. Then, there exists a poset $X$ and a $G$-grading on $I(X)$ such that
\[
I(X)\cong_G\left(\begin{array}{cc}\mathbb{F}H_1&M\\&\mathbb{F}H_2\end{array}\right),
\]
where $M\cong_G[(\chi_1,g_1),\ldots,(\chi_t,g_t)]$.
\end{Thm}

\begin{proof}
Since $H_{12}\hookrightarrow H_i$, we have $p_i:\widehat{H}_i\to\widehat{H}_{12}$.

Let $X=\widehat{H}_1\dot{\cup}\widehat{H}_2$, and, for each $j=1,2,\ldots,t$, define $\eta_1\precsim_j\eta_2$ if $\eta_1\in\widehat{H}_1$, $\eta_2\in\widehat{H}_2$ and $p_1(\eta_1)=\chi_jp_2(\eta_2)$. We define $\eta_1\precsim\eta_2$ if either $\eta_1=\eta_2$ or $\eta_1\precsim_j\eta_2$ for some $j$. Then $(X,\precsim)$ is a poset. Write
\[
M=\mathbb{F}H_1m_1\mathbb{F}H_2\oplus\cdots\oplus\mathbb{F}H_1m_t\mathbb{F}H_2,
\]
where each $m_i$ is homogeneous of degree $g_i$, and $hm_i=\chi_i(h)m_ih$, for all $h\in H_{12}$. We denote $\mathcal{A}=\left(\begin{array}{cc}\mathbb{F}H_1&M\\&\mathbb{F}H_2\end{array}\right)$, and define $\psi:\mathcal{A}\to I(X)$ via
\[
\psi\left(\begin{array}{cc}h_1\\&h_2\end{array}\right)=\left(\sum_{\eta_1\in\widehat{H}_1}\eta_1(h_1)e_{\eta_1}\right)+\left(\sum_{\eta_2\in\widehat{H}_2}\eta_2(h_2)e_{\eta_2}\right),
\]
and $\psi(hm_jk)=\sum_{\eta_1\precsim_j\eta_2}\eta_1(h)\eta_2(k)e_{\eta_1\eta_2}$, $h\in H_1$, $k\in H_2$. We need to show that $\psi$ is well-defined. For, it is enough to show that
\[
\psi\left(\begin{array}{cc}h&0\\&0\end{array}\right)\psi(m_j)=\chi_j(h)\psi(m_j)\psi\left(\begin{array}{cc}0&0\\&h\end{array}\right),
\]
for each $h\in H_{12}$. We have
\begin{align*}
\psi\left(\begin{array}{cc}h&0\\&0\end{array}\right)\psi(m_j)&=\left(\sum_{\eta_1\in\widehat{H}_1}\eta_1(h)e_{\eta_1}\right)\left(\sum_{\eta_1\precsim_j\eta_2}e_{\eta_1\eta_2}\right)=\sum_{\eta_1\in\widehat{H}_1}\left(\eta_1(h)\sum_{\eta_1\precsim\eta_2}e_{\eta_1\eta_2}\right)\\%
&=\sum_{\eta_1\in\widehat{H}_1}\sum_{\eta_1\precsim\eta_2}\underbrace{\eta_1(h)}_{(\chi_j\eta_2)(h)}e_{\eta_1\eta_2}=\sum_{\eta_2\in\widehat{H}_2}\sum_{\eta_1\precsim_j\eta_2}\chi_j(h)\eta_2(h)e_{\eta_1\eta_2}\\%
&=\chi_j(h)\psi(m_j)\psi\left(\begin{array}{cc}0&0\\&h\end{array}\right).
\end{align*}

Finally, we claim that $\psi$ is an algebra isomorphism. It is clear that $\psi$ is an algebra isomorphism when restricted to the diagonal part of $\mathcal{A}$. Moreover, the product of two strict upper elements of $\mathcal{A}$ is zero. By construction, $\psi$ satisfies $\psi(hm_jk)=\psi(h)\psi(m_j)\psi(k)$. Hence, $\psi$ is an isomorphism, and it induces a $G$-grading on $I(X)$ isomorphic to the grading on $\mathcal{A}$.
\end{proof}

As a consequence, we obtain a complete classification of the graded triangular algebras that can be realized as a graded incidence algebra.

\begin{Prop}\label{realization}
Let $\mathbb{F}$ be a field, $G$ an abelian group, $H_1$, $H_2\subseteq G$ be finite subgroups, $M$ a $G$-graded $(\mathbb{F}H_1,\mathbb{F}H_2)$-bimodule, and denote $M=\mathcal{V}\up$ (see \Cref{thm:gradedbimodule}), for some $\mathcal{V}\in\YD{H_{12}}{G}$. Then
$$
\left(\begin{array}{cc}\mathbb{F}H_1&M\\&\mathbb{F}H_2\end{array}\right)
$$
is realized as an incidence algebra endowed with a $G$-grading if, and only if, $\mathcal{V}$ is a $\mathbb{F}H_{12}$-submodule of the regular module $\mathbb{F}H_{12}$ and $\mathbb{F}$ is a splitting field for $H_1$ and $H_2$.
\end{Prop}
\begin{proof}
One direction is the previous theorem, and the other direction is proved in \cite[Theorems 1 and 2]{SSY}.
\end{proof}

\section{Product of bimodules\label{S_product}}

Throughout this section, let \( G \) be an abelian group, and let \( H_1 \), \( H_2 \), and \( H_3 \) be finite subgroups of \( G \). Define \( H_{123} = H_1 \cap H_2 \cap H_3 \), and for \( 1 \leq i < j \leq 3 \), define \( H_{ij} = H_i \cap H_j \).
We assume that $X$ is a poset and consider a $G$-grading on $I(X)$ such that
\[
I(X)\cong_G\left(\begin{array}{ccc}\mathbb{F}H_1&M_{12}&M_{13}\\&\mathbb{F}H_2&M_{23}\\&&\mathbb{F}H_3\end{array}\right),
\]
where $M_{ij}\ne0$ is a $G$-graded $(\mathbb{F}H_i,\mathbb{F}H_j)$-bimodule. Let $J(I(X))$ denote the Jacobson radical of the algebra $I(X)$. It is easy to see that
\[
J(I(X))=\left(\begin{array}{ccc}0&M_{12}&M_{13}\\&0&M_{23}\\&&0\end{array}\right),
\]

Furthermore, the powers of the Jacobson radical give a filtration on $I(X)$. The graded filtered algebra is isomorphic to the original one; that is, $I(X)\cong_G I(X)/J(I(X))\oplus J(I(X))/J(I(X))^2\oplus J(I(X))^2$. Moreover, it is not hard to see that $J(I(X))/J(I(X))^2\cong M_{12}\oplus M_{23}$ and $J(I(X))^2\cong M_{13}$. The product of the algebra gives a surjective map
\[
J(I(X))/J(I(X))^2\times J(I(X))/J(I(X))^2\to J(I(X))^2.
\]
Since $M_{12}^2=M_{23}^2=0$, one obtains a surjective $\mathbb{F}H_2$-balanced bilinear map $M_{12}\times M_{23}\to M_{13}$. Hence, the structure of $M_{13}$ is totally determined by the structures of $M_{12}$ and $M_{23}$. This discussion is another interpretation of \cite[Corollary 30]{SSY}. In this section, we will describe the structure of $M_{13}$ in terms of $M_{12}$ and $M_{23}$. As a consequence, we will be able to determine which graded algebras of the above kind are realized as a graded incidence algebra.

\begin{Lemma}\label{firstlemma}
For $i=1,2$, assume that $M_{i,i+1}\cong_G[(\chi_{i,i+1},g_{i,i+1})]$, and denote $M_{13}\cong_G[(\chi_k,g_k)\mid k=1,\ldots,t]$. Then $\chi_k|_{H_{123}}=\chi_{12}|_{H_{123}}\chi_{23}|_{H_{123}}$, for each $k=1,\ldots,t$.
\end{Lemma}
\begin{proof}
Consider the following commutative diagram, given by the restriction maps:
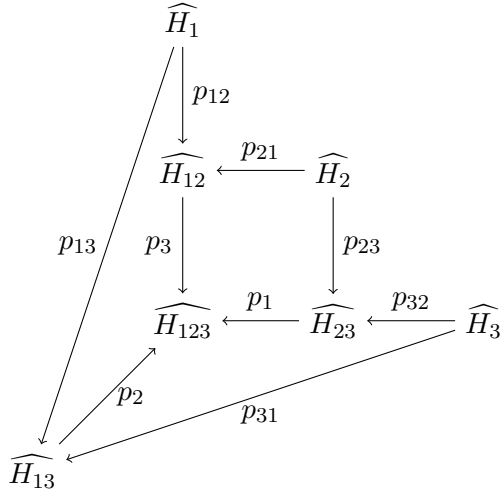
\begin{center}
\begin{tikzpicture}
    \node (1) at (0,0) {$\widehat{H_1}$};
    \node (12) at (0,-2) {$\widehat{H_{12}}$};
    \node (2) at (2,-2) {$\widehat{H_2}$};
    \node (23) at (2,-4) {$\widehat{H_{23}}$};
    \node (3) at (4,-4) {$\widehat{H_3}$};
    \node (123) at (0,-4) {$\widehat{H_{123}}$};
    \node (13) at (-2,-6) {$\widehat{H_{13}}$};
    \draw[->] (1) -- node[right] {$p_{12}$} (12);
    \draw[->] (2) -- node[above] {$p_{21}$} (12);
    \draw[->] (2) -- node[right] {$p_{23}$} (23);
    \draw[->] (3) -- node[above] {$p_{32}$} (23);
    \draw[->] (23) -- node[above] {$p_{1}$} (123);
    \draw[->] (12) -- node[left] {$p_3$} (123);
    \draw[->] (1) -- node[left] {$p_{13}$} (13);
    \draw[->] (3) -- node[below] {$p_{31}$}(13);
    \draw[->] (13) -- node[right] {$p_2$} (123);
\end{tikzpicture}
\end{center}

Let $\chi_i\in[M_{13}]$ and $\eta_1\in\widehat{H_1}$. Let $\eta_3\in\widehat{H_3}$ be such that $p_{13}(\eta_1)=\chi_ip_{31}(\eta_3)$. Then $\eta_1\precsim\eta_3$. Let $\eta_2\in\widehat{H_2}$  be such that $\eta_1\precsim\eta_2$ and $\eta_2\precsim\eta_3$. It means that one has
\[
p_{12}(\eta_1)=\chi_{12}p_{21}(\eta_2),\quad p_{23}(\eta_2)=\chi_{23}p_{32}(\eta_3).
\]
Applying $p_3$ to the first equation, we obtain
\[
\underbrace{p_3\circ p_{12}(\eta_1)}_{p_2\circ p_{13}(\eta_1)}=p_3(\chi_{12}p_{21}(\eta_2))=p_3(\chi_{12})\underbrace{p_3\circ p_{21}(\eta_2)}_{p_1\circ p_{23}(\eta_2)}=p_3(\chi_{12})p_1(\chi_{23})p_1\circ p_{32}(\eta_3).
\]
Since $p_2\circ p_{13}(\eta_1)=p_2(\chi_ip_{31}(\eta_3))$, we get $\chi_i|_{H_{123}}=(\chi_{12}|_{H_{123}})(\chi_{23}|_{H_{123}})$.
\end{proof}

We use the construction of the poset given in \Cref{construction}. So $X=\widehat{H_1}\dot\cup\widehat{H_2}\dot\cup\widehat{H_3}$. As in \cite{SSY}, given $\eta\in X$, we denote
\begin{align*}
\ell(\widehat{H_i},\eta)&=|\{\eta_i\in\widehat{H_i}\mid\eta_i\precsim\eta\}|,\\%
\ell(\eta,\widehat{H_j})&=|\{\eta_j\in\widehat{H_j}\mid\eta\precsim\eta_j\}|,\\%
\ell(\widehat{H_i},\widehat{H_j})&=|\{(\eta_i,\eta_j)\mid\eta_i\in\widehat{H_i},\eta_j\in\widehat{H_j},\,\eta_i\precsim\eta_j\}|.
\end{align*}
Note that, if $i<j$, then $\dim_\mathbb{F}M_{ij}=\ell(\widehat{H_i},\widehat{H_j})$. Moreover, \cite[Lemma 16]{SSY} proves that, for any $\eta_i\in\widehat{H_i}$ and $\eta_j\in\widehat{H_j}$, one has
\begin{equation}\label{link}
\ell(\widehat{H_i},\widehat{H_j})=|H_i|\ell(\eta_i,\widehat{H_j})=|H_j|\ell(\widehat{H_i},\eta_j).
\end{equation}
In addition, if $M_{ij}$ is irreducible, then $\dim_\mathbb{F}M_{ij}=\frac{|H_i||H_j|}{|H_{ij}|}$ (\cite[Theorem 2]{SSY}).

We start proving the result by imposing a particular restriction.
\begin{Lemma}\label{preliminarycase}
For $i=1,2$, assume that $M_{i,i+1}\cong_G[(\chi_{i,i+1},g_{i,i+1})]$. If 
\[
\ell (\widehat{H_1}, \widehat{H_3}) \ge \frac{|H_1||H_3|}{|H_{123}|},
\]
then
\[
M_{13}\cong_G[(\chi,g_{12}g_{23})\mid%
\chi|_{H_{123}} = \chi_{12}|_{H_{123}} \chi_{23}|_{H_{123}}].
\]
\end{Lemma}
\begin{proof}
As remarked above, and from the hypothesis, one has
\[
\dim_\mathbb{F}M_{13}=\ell(\widehat{H_1},\widehat{H_3})\ge\frac{|H_1||H_3|}{|H_{123}|}.
\]
On the other hand, we can decompose $M_{13}=M_1\oplus\cdots\oplus M_t$, as a sum of graded-irreducible bimodules. From \cite[Theorem 2]{SSY}, one has $\dim_\mathbb{F}M_i=\frac{|H_1||H_3|}{|H_{13}|}$, for each $i$. Hence,
\[
\frac{|H_1||H_3|}{|H_{123}|}\le\dim_\mathbb{F}M_{13}=t\frac{|H_1||H_3|}{|H_{13}|}.
\]
It means that $t\ge\frac{|H_{13}|}{|H_{123}|}$. On the other hand, from \Cref{firstlemma}, each $\chi\in[M_{13}]$ belongs to the set $\{\chi\in\widehat{H_{13}}\mid\chi|_{H_{123}}=\left(\chi_{12}|_{H_{123}}\right)\left(\chi_{23}|_{H_{123}}\right)\}$, so $t\le|\{\chi\in\widehat{H_{13}}\mid\chi|_{H_{123}}=\left(\chi_{12}|_{H_{123}}\right)\left(\chi_{23}|_{H_{123}}\right)\}|$. However, the above set is a translation of the kernel of the restriction map $\widehat{H_{13}}\to\widehat{H_{123}}$. Hence,
\[
|\{\chi\in\widehat{H_{13}}\mid\chi|_{H_{123}}=\left(\chi_{12}|_{H_{123}}\right)\left(\chi_{23}|_{H_{123}}\right)\}|=\frac{|H_{13}|}{|H_{123}|}\le t.
\]
Thus, the above is an equality, and every element of $\{\chi\in\widehat{H_{13}}\mid\chi|_{H_{123}}=\left(\chi_{12}|_{H_{123}}\right)\left(\chi_{23}|_{H_{123}}\right)\}$ should appear in $[M_{13}]$.
\end{proof}

As a consequence, we obtain:

\begin{Lemma}\label{case1}
If $H_3\subseteq H_2$, then
\begin{align*}
M_{13}\cong_G[(\chi,(\deg\chi_{12})(\deg\chi_{23}))\mid\, &
\chi|_{H_{123}} = (\chi_{12}|_{H_{123}})(\chi_{23}|_{H_{123}}),\\&\chi_{12}\in[M_{12}],\chi_{23}\in[M_{23}]].
\end{align*}
\end{Lemma}
\begin{proof}
From the bilinearity of the product and \Cref{firstlemma}, we may assume that $M_{i,i+1}\cong_G[(\chi_{i,i+1},g_{i,i+1})]$, for $i=1,2$. Since $H_3\subseteq H_2$, one has $H_{123}=H_{13}$. Let $M\subseteq M_{13}$ be a graded-irreducible bimodule. From the previous discussion, we obtain
\[
\dim_\mathbb{F}M_{13}\ge\dim_\mathbb{F}M=\frac{|H_1||H_3|}{|H_{13}|}=\frac{|H_1||H_3|}{|H_{123}|}.
\]
Hence, the result follows from \Cref{preliminarycase}.
\end{proof}

Now, we will investigate another situation where the result is valid.
\begin{Lemma}\label{case2}
If $H_2\subseteq H_3$, then
\begin{align*}
M_{13}\cong_G[(\chi,(\deg\chi_{12})(\deg\chi_{23}))\mid &
\chi|_{H_{123}} = (\chi_{12}|_{H_{123}})(\chi_{23}|_{H_{123}}),\\&\chi_{12}\in[M_{12}],\chi_{23}\in[M_{23}]].
\end{align*}
\end{Lemma}
\begin{proof}
Since $H_2\subseteq H_3$, one has $H_{12}=H_{123}$. Again, we may assume that $M_{i,i+1}\cong_G[(\chi_{i,i+1},g_{i,i+1})]$, for $i=1,2$. For $i=2,3$, denote by $p_i:\widehat{H_i}\to\widehat{H_{23}}$ the restriction map. Since $H_3\subseteq H_2$, one has $H_{23}=H_2$; so $\widehat{H_2}=\widehat{H_{23}}$. Then, $p_2$ is the identity map. If $\eta_3\in\widehat{H_3}$ and $\eta_2$, $\eta_2'\in\widehat{H_2}$ are such that $\eta_2\precsim\eta_3$ and $\eta_2'\precsim\eta_3$, then
\[
\eta_2=p_2(\eta_2)=\chi_{23}p_3(\eta_3)=p_2(\eta_2')=\eta_2'.
\]
It means that $\ell(\widehat{H_2},\eta_3)=1$. Thus, \eqref{link} gives $\ell(\eta_2,\widehat{H_3})=|H_3|/|H_2|$. Now, for each $\eta_1\precsim\eta_3$ ($\eta_i\in\widehat{H_i}$), there exists a unique $\eta_2\in\widehat{H_2}$ such that $\eta_1\precsim\eta_2\precsim\eta_3$. Hence,
\begin{align*}
\ell(\widehat{H_1},\widehat{H_3})&=\sum_{\eta_2\in\widehat{H_2}}\ell(\widehat{H_1},\eta_2)\ell(\eta_2,\widehat{H_3})=\ell(\widehat{H_1},\widehat{H_2})\ell(\eta_2,\widehat{H_3})\\
&=\frac{|H_1||H_2|}{|H_{12}|}\frac{|H_3|}{|H_2|}=\frac{|H_1||H_3|}{|H_{123}|}
\end{align*}
The result follows from \Cref{preliminarycase}.
\end{proof}

Next, we need to construct an auxiliary result.

\begin{Lemma}\label{increasedposet}
Let $G$ be an abelian group, $H_1$, $H_2\subseteq G$ finite subgroups, and $M$ a $G$-graded $(\mathbb{F}H_1,\mathbb{F}H_2)$-bimodule, where the characters that appear in the parametrization of $M$ are pairwise distinct. Then, there exists a poset $X'$ and a $G$-grading on $X'$ such that
\[
I(X')\cong_G\left(\begin{array}{ccc}\mathbb{F}H_1&M'&M\\&\mathbb{F}H_{12}&\mathds{1}\\&&\mathbb{F}H_2\end{array}\right),
\]
where $M'$ is the $G$-graded $(\mathbb{F}H_1,\mathbb{F}H_{12})$-bimodule having the same characters as $M$, and $\mathds{1}\cong_G[(1,1)]$ as $G$-graded $(\mathbb{F}H_{12},\mathbb{F}H_2)$-bimodules, where the first $1:H_{12}\to\mathbb{F}^\times$ denotes the trivial character.
\end{Lemma}

\begin{proof}
We let $X'=\widehat{H_1}\dot\cup\widehat{H_{12}}\dot\cup\widehat{H_2}$. Let $\widehat{K_1}=\widehat{H_1}$, $\widehat{K_2}=\widehat{H_{12}}$ and $\widehat{K}_3=\widehat{H_2}$. For $i<j$, we will denote
\[
M_{ij}=\mathrm{Span}\{e_{\eta_i\eta_j}\mid\eta_i\in\widehat{K_i},\eta_j\in\widehat{K_j},\,\eta_i\precsim\eta_j\}.
\]
We give the $G$-grading on $I(X')$ in such a way that the diagonal part of $I(X')$ is $\mathbb{F}H_1\oplus\mathbb{F}H_{12}\oplus\mathbb{F}H_2$. Applying \Cref{construction} twice we can construct the poset structure on $X'$ in such a way that $M_{12}=M'$ and $M_{23}=\mathds{1}$, as $G$-graded bimodules. Now, \Cref{case2} guarantees that we can impose a structure of $G$-graded bimodule on $M_{13}$ so that it coincides with $M$. Again, from \Cref{case2}, the given $G$-grading on $I(X')$ is a well-defined structure of $G$-graded algebra.
\end{proof}

Now, we are in a position to state the main result of this section.

\begin{Thm}\label{mainresult}
Let $G$ be an abelian group, and let $H_1$, $H_2$, $H_3\subseteq G$ finite subgroups, and for $i\le j$, let $H_{ij}=H_i\cap H_j$ and $M_{ij}$ be a $G$-graded $(\mathbb{F}H_i,\mathbb{F}H_j)$-bimodule. Let $X$ be a poset and consider a $G$-grading on $I(X)$ such that
\[
I(X)\cong_G\left(\begin{array}{ccc}\mathbb{F}H_1&M_{12}&M_{13}\\&\mathbb{F}H_2&M_{23}\\&&\mathbb{F}H_3\end{array}\right).
\]

Then $\chi\in[M_{13}]$ if, and only if, there exist $\chi_{12}\in[M_{12}]$ and $\chi_{23}\in[M_{23}]$ such that $\chi|_{H_{123}}=\left(\chi_{12}|_{H_{123}}\right)\left(\chi_{23}|_{H_{123}}\right)$. In each of these cases, one has $\deg\chi=(\deg\chi_{12})(\deg\chi_{23})$.
\end{Thm}
\begin{proof}
As before, we may assume, without loss of generality, that $M_{i,i+1}\cong_G[(\chi_{i,i+1},g_{i,i+1})]$, for $i=1$ and $i=2$. We replace the poset $X$ with a larger poset $X'$, as in \Cref{increasedposet}, in such a way that
\[
I(X')\cong_G\left(\begin{array}{cccc}\mathbb{F}H_1&M_{12}&N&M_{13}\\&\mathbb{F}H_2&M_{23}'&M_{23}\\&&\mathbb{F}H_{23}&\mathds{1}\\&&&\mathbb{F}H_3\end{array}\right).
\]
Since $H_{23}\subseteq H_2$, \Cref{case1} describes the structure of $N$ as the extensions of the product of the characters of $M_{12}$ and $M_{23}'$. Restricting to the graded subalgebra
\[
\left(\begin{array}{ccc}\mathbb{F}H_1&N&M_{13}\\&\mathbb{F}H_{23}&\mathds{1}\\&&\mathbb{F}H_3\end{array}\right),
\]
then \Cref{case2} characterizes the structure of $M_{13}$ as the extensions of the characters of $N$.
\end{proof}

\section{Final remarks}
We summarize the results obtained so far. Thus, we obtain the following structure of finite-dimensional graded algebras that may be realized as an incidence algebra endowed with a group grading.

\begin{Thm}\label{realizationgradedalgebra}
Let $\mathbb{F}$ be a field, $G$ an abelian group, $H_1,\ldots,H_t\subseteq G$ finite subgroups, for $i<j$, let $M_{ij}$ be a $G$-graded $(\mathbb{F}H_i,\mathbb{F}H_j)$-bimodule (which can be $0$). Denote $H_{ij}=H_i\cap H_j$ and $H_{ijk}=H_i\cap H_j\cap H_k$, and denote $M_{ij}=\mathcal{V}_{ij}\up$, where $\mathcal{V}_{ij}\in\YD{H_{ij}}{G}$. Let
\[
\mathcal{A}=\left(\begin{array}{cccc}\mathbb{F}H_1&M_{12}&\ldots&M_{1t}\\&\mathbb{F}H_2&\ddots&\vdots\\&&\ddots&M_{t-1,t}\\&&&\mathbb{F}H_t\end{array}\right),
\]
Then, there exists a poset $X$ and a $G$-grading on $I(X)$ such that $I(X)\cong_G\mathcal{A}$ if and only if:
\begin{enumerate}
\item For each $i=1,\ldots,t$, $\mathrm{char}\,\mathbb{F}$ does not divide $|H_i|$ and $\mathbb{F}$ contains a primitive $\mathrm{exp}\,H_i$-root of $1$,
\item $\mathcal{V}_{ij}$ is a $\mathbb{F}H_{ij}$-submodule of the regular module $\mathbb{F}H_{ij}$,
\item For each $i<j<k$ such that $M_{ij}\ne0$ and $M_{jk}\ne0$, one has:
\begin{enumerate}
\item for each $\chi\in[M_{ik}]$, there exists $\chi_{ij}\in[M_{ij}]$ and $\chi_{jk}\in[M_{jk}]$ such that $\chi|_{H_{ijk}}=(\chi_{ij}|_{H_{ijk}})(\chi_{jk}|_{H_{ijk}})$ and $\deg\chi=(\deg\chi_{ij})(\deg\chi_{jk})$,
\item for each $\chi\in\widehat{H_{ik}}$ such that there exists $\chi_{ij}\in[M_{ij}]$ and $\chi_{jk}\in[M_{jk}]$ satisfying $\chi|_{H_{ijk}}=(\chi_{ij}|_{H_{ijk}})(\chi_{jk}|_{H_{ijk}})$, one has $\chi\in[M_{ik}]$.
\end{enumerate}
\end{enumerate}
\end{Thm}
\begin{proof}
Condition (1) is necessary and sufficient for $\mathbb{F}$ to be a splitting field for each group algebra $\mathbb{F}H_i$. Condition (2) provides a construction of the graded bimodules above the diagonal (\Cref{realization}). Condition (3) ensures that the grading is consistent with a group grading on an incidence algebra (\Cref{mainresult}).

Conversely, the conditions (1) and (2) are obtained in \cite[Theorems 1 and 2]{SSY}. The condition (3) is due to \Cref{mainresult}.
\end{proof}

Given a poset $(\mathcal{E},\le)$, recall that we say that $y$ \emph{covers} $x$ if $x\le y$, $x\ne y$, and $x\le z\le y$ implies that either $x=z$ or $z=y$. In this case, we denote $x\lessdot y$.

We will parameterize a group grading on an incidence algebra by a triple $((\mathcal{E},\le),\{H_i\}_{i\in\mathcal{E}},\{M_{ij}\}_{i\lessdot j})$, where $(\mathcal{E},\le)$ is a poset, each $H_i$ is a finite abelian subgroup of $G$ and $M_{ij}$ is a $G$-graded $(\mathbb{F}H_i,\mathbb{F}H_j)$-bimodule. The characters of $M_{ij}$ are pairwise distinct. Let $\mu_i\in\widehat{H_i}$, $\mu_j\in\widehat{H_j}$ and denote $M_{ij}\cong_G[(\chi_1,g_1),\ldots,(\chi_t,g_t))]$. We define $\mu_i M_{ij}\mu_j$ as the $G$-graded $(\mathbb{F}H_i,\mathbb{F}H_j)$-bimodule
\[
\mu_i M_{ij}\mu_j\cong_G[((\mu_i|_{H_{ij}})\chi_1(\mu_j|_{H_{ij}}),g_1),\ldots,((\mu_i|_{H_{ij}})\chi_t(\mu_j|_{H_{ij}}),g_t)].
\]

Although there is no restriction on the characters that may appear, the degrees must be consistent with condition (3) of \Cref{realizationgradedalgebra}; i.e., if $i\lessdot j\lessdot\ell$ and $i\lessdot k\lessdot\ell$, then for each equation $\chi_{ij}|_H\chi_{j\ell}|_H=\chi_{ik}|_H\chi_{k\ell}|_H$, where $H=H_i\cap H_j\cap H_k\cap H_\ell$, one has $\deg(\chi_{ij})\deg(\chi_{j\ell})=\deg(\chi_{ik})\deg(\chi_{k\ell})$.

The isomorphism classes of group gradings on incidence algebras are known. We state the result using our terminology for completeness.

\begin{Thm}\label{isoclasses}
Let $G$ be an abelian group, and let $\Gamma$ and $\Gamma'$ be $G$-gradings on finite-dimensional incidence algebras, parameterized by $((\mathcal{E},\le),\{H_i\}_{i\in\mathcal{E}},\{M_{ij}\}_{i\lessdot j})$ and $((\mathcal{E}',\le),\{H_i'\}_{i\in\mathcal{E}'},\{M_{ij}'\}_{i\lessdot j})$, respectively. Then $\Gamma\cong\Gamma'$ if, and only if, there exists a poset isomorphism $\alpha:\mathcal{E}\to\mathcal{E}'$ and a set of characters $(\mu_i)_{i\in\mathcal{E}}$, $\chi_i\in\widehat{H_i}$, such that:
\begin{enumerate}
\item $H_i'=H_{\alpha(i)}$, for each $i\in\mathcal{E}$,
\item $M_{ij}\cong_G \mu_i(M_{\alpha(i),\alpha(j)})\mu_j$, for each $i\lessdot j$ (see \Cref{isogradedbimodules}).
\end{enumerate}
\end{Thm}
\begin{proof}
This is a restatement of \cite[Theorem 34]{SSY}.
\end{proof}

\end{document}